\documentclass[12pt]{article}

\RequirePackage{kpfonts}
\RequirePackage[sf,bf,small,raggedright]{titlesec}
\usepackage{textcomp}
\usepackage{amssymb}
\usepackage{bbm}
\usepackage{fancyhdr}
\usepackage{amsmath}
\usepackage{amsbsy,amsthm}
\usepackage{amscd}
\usepackage{framed}
\usepackage{latexsym}
\usepackage{graphicx}   
\usepackage{pdfsync}
\usepackage{blkarray}
\usepackage{multirow}
\usepackage{hyperref}
\usepackage{tcolorbox}
\usepackage[titletoc]{appendix}

\newcommand{\R}{\mathbb{R}}

\newcommand{\inr}[1]{\left\langle #1 \right\rangle}

\newcommand{\E}{\mathbb{E}}

\newtheorem{lemma}{Lemma}
\newtheorem{theorem}{Theorem}
\newtheorem{corollary}{Corollary}
\newtheorem{proposition}{Proposition}
\newtheorem*{proposition*}{Proposition}
\newtheorem{example}{Example}
\newtheorem{definition}{Definition}

\newtheorem{remark}{Remark}

\numberwithin{equation}{section}

\def\IND{\mathbbm{1}}

\newcommand{\ol}{\overline}
\newcommand{\wt}{\widetilde}
\newcommand{\wh}{\widehat}

\newcommand{\EXP}{\mathbb{E}}
\newcommand{\PROB}{\mathbb{P}}

\newcommand{\var}{\mathrm{Var}}
\newcommand{\F}{{\mathcal F}}

\newcommand{\argmax}{\mathop{\mathrm{argmax}}}

\usepackage[authoryear]{natbib}
\newtheorem*{assumption*}{\assumptionnumber}
\providecommand{\assumptionnumber}{}
\makeatletter
\newenvironment{assumption}[1]
 {%
  \renewcommand{\assumptionnumber}{Assumption $#1$}%
  \begin{assumption*}%
  \protected@edef\@currentlabel{$#1$}%
 }
 {%
  \end{assumption*}
 }
\makeatother

\begin{document}

\title{On Mean Estimation for Heteroscedastic Random Variables
\thanks{
G\'abor Lugosi was supported by
the Spanish Ministry of Economy and Competitiveness,
Grant MTM2015-67304-P and FEDER, EU;
and Google Focused Award ``Algorithms and Learning for AI''. Luc
Devroye was supported by NSERC Discovery Grants and by an FRQNT Team Research Grant.
}}
\author{
Luc Devroye 
\thanks{School of Computer Science, McGill University, Montreal, Canada. Email: lucdevroye@gmail.com}
\and 
Silvio Lattanzi
\thanks{Google Research, Z\"{u}rich, Email: \{silviol, zhivotovskiy\}@google.com}
\and G\'abor Lugosi
\thanks{Department of Economics and Business, Pompeu
  Fabra University, Barcelona, Spain. Email: gabor.lugosi@upf.edu}
\thanks{ICREA, Pg. Lluís Companys 23, 08010 Barcelona, Spain}
\thanks{Barcelona Graduate School of Economics}
\and
Nikita Zhivotovskiy \footnotemark[3]
}

\maketitle
\abstract{
We study the problem of estimating the common mean $\mu$ of $n$
independent symmetric random variables with different and unknown standard deviations $\sigma_1 \le \sigma_2 \le \cdots \le\sigma_n$.
We show that, under some mild regularity assumptions on the distribution, 
there is a fully adaptive estimator $\wh{\mu}$ such that it is invariant to permutations of the elements of the sample and satisfies that, up to logarithmic factors, with high probability,
\[
|\wh{\mu} - \mu| \lesssim \min\left\{\sigma_{m^*}, \frac{\sqrt{n}}{\sum_{i = \sqrt{n}}^n \sigma_i^{-1}} \right\}~,
\]
where the index $m^* \lesssim \sqrt{n}$ satisfies $m^* \approx
\sqrt{\sigma_{m^*}\sum_{i = m^*}^n\sigma_i^{-1}}$. 
}

\section{Introduction}

In this note we study the problem of estimating the common mean $\mu
\in \R$ of $n$ independent real random variables $X_1,\ldots,X_n$. These random variables do not need
to be identically distributed. Moreover, the variances of the $X_i$
may greatly vary and therefore the information each observation
carries about the mean may be different. For the sake of this
introductory discussion, assume that the $X_i$ all have  normal
distribution so that $X_i \sim \mathcal{N}(\mu, \sigma_i^{2})$ for
some $0< \sigma_1\le \cdots \le \sigma_n$.

If the values of the standard deviations $\sigma_i$ were known, then
one could choose the maximum likelihood estimator 
\[
    \wh{\mu}= \frac{\sum_{i=1}^n \frac{X_i}{\sigma_i^2}}{\sum_{i=1}^n \frac{1}{\sigma_i^2}}~,
\] 
leading to an expected error $\EXP |\wh{\mu} - \mu| \le
\left(\sum_{i=1}^n \sigma_i^{-2}\right)^{-1/2}$.
The general study of such estimators goes back to \citep[Chapter 3,
Section 4]{ibragimov2013statistical} where the estimation of a single parameter based on 
independent but non-identically distributed observations is studied.
However, this idealistic estimator assumes that the standard deviation
of each sample point is known to the statistician. 

In this note we consider the situation where nothing is known about the
values of the $\sigma_i$ (or their assignments to the data points). In particular, we focus on the estimators invariant to permutations of the elements of the sample.
Naturally, one may always compute the sample mean $(1/n)\sum_{i = 1}^n
X_i$. However, the sample mean has an error of the order of $(1/n)
\left(\sum_{i=1}^n \sigma_i^2\right)^{1/2}$ whose performance 
deteriorates even if a single data point has a large variance.

For symmetric distributions like the normal distribution, another -- and
more robust -- natural
estimator of the mean is the sample median. 
One of the contributions of this note is to provide new
non-asymptotic performance guarantees for the sample median.
In particular, we show that under some mild assumptions
the error of the sample median is bounded, with high probability, by 
\begin{equation}
\label{eq:harmonicmean}
     \frac{c \sqrt{n\log n }}{\sum_{i = \sqrt{c n\log n}}^n \sigma_i^{-1}} 
\end{equation}
for some constant $c$ (see Proposition \ref{prop:medianintterval} for
the rigorous statement). 

As simple as the sample median is, it has the disadvantage that it does not take advantage
of the presence of data points with very small variance. Indeed, the
performance of the sample median is essentially insensitive to the approximately
$\sqrt{n}$ smallest variances. At the same time, the presence of
data with very small variances makes the problem much easier. 
A simple way to exploit such situations is in using the so-called 
\emph{modal interval estimator} introduced by
\cite{chernoff1964estimation} for estimating the mode of a density
function. The modal interval estimator looks for the most populated
interval of a certain length $s>0$ and outputs its mid-point. The main
challenge in applying this method in our setting is that without any knowledge
of the variances $\sigma_1,\sigma_2,\dots$ it is a hard to establish a
good value of $s$ a priori. In Proposition 
\ref{thm:intervalest} below we establish a simple sufficient condition
for the length $s$ that guarantees that the modal interval contains the mean $\mu$. Roughly speaking, 
this condition guarantees that random fluctuations of the data far
from the mean cannot produce an interval of length $s$ that has more
points than the expected number of points in the interval of same
length centered at the mean $\mu$. We call such ``good'' values of $s$
\emph{admissible}. Admissibility of an interval length depends, in a
complex way, on the entire sequence $\sigma_1,\ldots,\sigma_n$.
Ideally, one would like to use the modal interval estimator with the
smallest possible admissible interval length. The main contribution
of this note is a fully adaptive estimator that essentially achieves
this goal. More precisely, without any previous knowledge of the $\sigma_i$, 
 we show that one can construct a completely data-driven estimator that has a
performance at least as good (up to constant factors in the error) 
as the best of the sample median and the
modal interval estimator with the smallest admissible interval
length. 

In the remainder of this introduction we discuss previous related work.
In Section \ref{sec:median} the analysis of the sample median is
presented. We also show that an appropriately chosen \emph{median
  interval} is a valid empirical confidence interval. This is
important in the construction of the adaptive estimator.
The modal interval estimator is analyzed in Section \ref{sec:modalinterval}.
The fully adaptive estimator is described in Section
\ref{sec:adaptiveestimator} and its performance guarantees are
established in Theorem \ref{thm:adaptiveestimator}.
In Section \ref{sec:comparisons} we take a closer look at  some
concrete examples and compare our performance bounds with those 
of previous work. 

\subsubsection*{Related work}

For some classical references on the maximum likelihood
estimator in our setup we refer to the work of
\citet{ibragimov1976local} and
\citet{beran1982robust}. The sample median has been analyzed in the literature in our setup. For example, 
\citet{mizera1998necessary} provide necessary and sufficient
conditions for the consistency of the sample median for triangular
arrays of independent, not identically distributed random variables
(in a more general setting than ours).
The role of the sum of the reciprocals of the standard deviations as in \eqref{eq:harmonicmean}
appears in early work. In particular, the result of \citet[Theorem 2]{nevzorov1984rate} can be used to provide
rates of convergence of the sample median to the normal law for 
non-identically distributed Gaussian data  that involves 
$\sum_{i=1}^n \sigma_i^{-1}$. 
The work of \citet*[Theorem 7]{gordon2006minimum} uses this quantity in
the context  of the moments
of order statistics for non-identically distributed random
variables. Moreover, the work of \citet[Corollary 6]{xia2019non} makes direct
connections between the sum of reciprocals and the performance of the
sample median, see Section \ref{sec:comparisons} for a detailed comparison.
The same quantity appears in the analysis of the iterative trimming
algorithm of \citet[Remark following Theorem 1]{yuan2020learning}.
Importantly, in the context of the mean estimation problem, some of the
above-mentioned results provide performance guarantees when the value
of $\sum_{i=1}^n \sigma_i^{-1}$ is large, whereas our bounds provide sharp
guarantees for the entire range of values of the sum of reciprocals of
the standard deviations.

The most related papers are \citep*{chierichetti2014learning}, 
\citep*{pensia2019estimating}, and 
\citep{yuan2020learning}. For example, \cite{chierichetti2014learning} construct
an estimator whose error is bounded, with high
probability, by
$\wt{O}\left(\sqrt{n}\sigma_{\log n}\right)$, where $\wt{O}(\cdot )$ suppresses multiplicative poly-logarithmic factors.
The \emph{hybrid} estimator of \citet[Proposition 5]{pensia2019estimating}
uses a combination of the \emph{shortest gap} with the median estimators, quite similar to
our estimator. 
However, in contrast to these previous results, the estimator proposed
here is fully adaptive and thus requires no tuning parameters.  
Our estimator also compares favourably with the iterative trimming
algorithm of \citet{yuan2020learning}, which does not cover the entire
range of the values of $\sigma_1, \ldots, \sigma_n$ and depends on
some tuning parameters and the initialization.  Section
\ref{sec:comparisons} includes extensive comparisons with these
papers. 
In particular, we show that, up to logarithmic factors,
our bounds are never worse than the previous (non-adaptive) bounds.

\subsubsection*{Notation}
In what follows, we denote $a \wedge b = \min\{a, b\}$ and $a \lor b = \max\{a, b\}$. Given $X_{1}, \ldots, X_{n}$ let $X_{(1)}, \ldots, X_{(n)}$ denote the non-decreasing rearrangement of its elements. The value $X_{(i)}$ is usually referred to as the $i$-\emph{th order statistic}. In what follows, $a \lesssim b$ and $b \gtrsim a$ denote the existence of a numerical constant $c$ such that $a \leq cb$.  The numerical constants are denoted by $c, c_{1}, c_{2}, \ldots > 0$. Their values may change from line to line. We also use the standard $O(\cdot )$ notation as well as its version $\wt{O}(\cdot )$ that suppresses multiplicative poly-logarithmic factors. Finally, let $[n]$ denote the set $\{1, \ldots, n\}$.

\section{Analysis of the $\alpha$-median interval}
\label{sec:median}

When the distribution of each random variable $X_i$ is symmetric about
the mean $\mu$,
the empirical median is a natural estimator of the
mean. In this section we present an analysis of the empirical median.
We assume the following regularity conditions.

\begin{assumption}{\mathbf{A}}
\label{AssumptionA}
Let $X_1, \ldots, X_n$ be independent random variables 
and let $0 < \sigma_1 \le \cdots \le \sigma_n$. We assume that
\begin{itemize}
\item[(i)] $\EXP X_i = \mu$ for all $i\in [n]$~;
\item[(ii)] Symmetry: for each $i\in [n]$, $X_i - \mu$ and $\mu - X_i$ have the same distribution~;
\item[(iii)] Tail assumption: for some constant $\beta > 0$,  we have that for any $t > 0$,
\begin{equation}
\label{eq:tailassumpt}
\PROB\left\{|(X_i - \mu)/\sigma_i| \ge t\right\} \le \exp(-\beta t)~.
\end{equation}
\end{itemize}
\end{assumption}

A canonical example satisfying Assumption \ref{AssumptionA} is when
 $X_i \sim \mathcal{N}(\mu, \sigma_i^2)$. In this case one  may choose
 $\beta = \sqrt{\frac{2}{\pi}}$.
Note that we do not need to assume that the $(X_i-\mu)/\sigma_i$ are
identically distributed. It suffices that they are independent,
symmetric, and satisfy the tail assumption \eqref{eq:tailassumpt}. Note also that condition
(iii) implies that $\PROB\{ |(X_i - \mu)/\sigma_i| < t \}$ is lower
bounded by $\beta t/2$ for $t\le 2/\beta$. 
In particular, if $X_i$ has an absolute continuous distribution, this assumption
implies that the density of $(X_i - \mu)/\sigma_i$ is bounded away
from zero near zero.

For reasons that will become apparent later, we consider not only the
empirical median as a point estimator but also the so-called \emph{median
  interval}, defined as the interval whose endpoints are $X_{(n/2-k)}$
and $X_{(n/2+k)}$ for an appropriately chosen value of $k$.
This will allow us to obtain an empirical confidence interval 
that is essential for our adaptive procedure.

To define the median interval, assume, for simplicity, that $n$ is
even and recall that $X_{(1)} \le X_{(2)} \le \cdots \le X_{(n)}$ denote the
order statistics of
$X_1,\ldots,X_n$. In order to avoid complications arising from ties,
we assume that the $X_i$ have a nonatomic distribution. 

We fix $\alpha \in (0,\sqrt{n}/2)$ such that $\alpha \sqrt{n}$ is
an integer. 
Consider the random interval 
\begin{equation}
\label{eq:medianinterval}
I_{\alpha} = [X_{(n/2 - \alpha \sqrt{n})}, X_{(n/2 + \alpha \sqrt{n})}]~.
\end{equation}
We refer to $I_{\alpha}$ as the $\alpha$-\emph{median interval}. 
Our first result
provides two key properties of the median interval: if $\alpha$ is
proportional to $\sqrt{\log(1/\delta)}$, the interval $I_{\alpha}$
contains the mean $\mu$ with probability at least $1-\delta$.
Moreover, we provide an upper bound for the length of $I_{\alpha}$
in terms of the sum of the reciprocals of the standard deviations.

\begin{proposition}
\label{prop:medianintterval}
Let Assumption \ref{AssumptionA} be satisfied. Fix $\delta \in (0, 1)$
such that $128\log \frac {6}{\delta} \le n$ and set $\alpha =
\sqrt{2\log \frac {6}{\delta}}$. The median interval $I_{\alpha}$ satisfies, with probability at least $1 - \delta$, that $\mu \in I_{\alpha}$ and
\[
|I_{\alpha}| \le 8e\sqrt{2}\left(\log\frac{3}{\delta} \lor \log\left(n+ 1\right)\right)\beta^{-1}\max\limits_{1 \le j \le 8\alpha\sqrt{n}}\frac{8\alpha\sqrt{n} + 1 - j}{\sum_{i = j}^n\sigma_i^{-1}}~.
\]
\end{proposition}

Note that by ignoring constant factors, Proposition \ref{prop:medianintterval} implies
\begin{equation}
\label{eq:simplemedianintreval}
|I_{\alpha}| \lesssim \beta^{-1}\log\left(\frac{n}{\delta}\right)\frac{\alpha\sqrt{n}}{\sum_{i = 8\alpha\sqrt{n}}^n\sigma_i^{-1}}~.
\end{equation}

The key to the proof of Proposition \ref{prop:medianintterval} is following rearrangement inequality due to \citet[Theorem 7]{gordon2006minimum}. Let $|X|_{(1)}, \ldots, |X|_{(n)}$ denote the non-decreasing rearrangement of $|X_{1}|, \ldots, |X_{n}|$.
\begin{lemma}
\label{lemmamedfirst}
Let $X_1, \ldots, X_n$ be independent random variables such
that for $0< \sigma_1 \le \cdots \le \sigma_n$ and $\beta > 0$, for
all $t > 0$, 
$\PROB\left(|X_i/\sigma_i| \ge t\right) \le \exp(-\beta t)$. Then for all $p \ge 1$ and $1\le k \le n$,
\[
\left(\E(|X|_{(k)})^p\right)^{\frac{1}{p}} 
\le 4\sqrt{2}\max\{p, \log(k + 1)\}\beta^{-1}\max\limits_{1 \le j \le k}\frac{k + 1 - j}{\sum_{i = j}^n\sigma_i^{-1}}~.
\]
\end{lemma}

\medskip
\noindent
{\bf Proof of Proposition \ref{prop:medianintterval}.}
First, we show that $\mu \in I_{\alpha}$. Without loss of generality,
we may assume that $\mu = 0$ for the rest of the proof. Let
$\varepsilon_{1}, \ldots, \varepsilon_n$ be independent Rademacher
random variables. Since the distribution of each $X_i$ is assumed to
be symmetric, $(\varepsilon_1|X_1|, \ldots, \varepsilon_n|X_n|)$ has
the same distribution as $(X_1, \ldots, X_n)$.  
Conditioning on the $X_{1}, \ldots, X_n$, we have, by
Hoeffding's inequality,
\[
\PROB\left(\mu \notin I_{\alpha}\right) = \PROB\left(\left|\sum_{i = 1}^n\varepsilon_i\right| > \alpha\sqrt{n}\right) \le 2\exp\left(-\frac{\alpha^2}{2}\right).
\]
We denote the event that $\mu \in I_{\alpha}$ by $E_1$ and proceed
with the bound on the length of the interval $|I_{\alpha}|$. Fix
$k \le n$ and consider $|X|_{(1)}, \ldots, |X|_{(k)}$ ---
these are the absolute values of the $k$ observations closest to
$\mu = 0$. Note that, depending on the realizations of the
random signs $\varepsilon_i$, the corresponding values
$\varepsilon_1|X_i|$ may be on either side of $\mu = 0$. Let $E_2$ be the event
that there are more than $k/4$ of these $k$ observations on both sides of
$\mu$. By a simple binomial estimate, 
\[
\PROB(E_2) \ge 1 -  2\exp\left(-\frac{k}{8}\right)~.
\]
Consider the event $E_1 \cap E_2$ and choose $k = 8\alpha\sqrt{n}$ so that at least $2\alpha\sqrt{n} + 1$ of these closest observations are on both sides of
$\mu$. On
this event since $I_\alpha$ contains $\mu = 0$ and exactly
$2\alpha\sqrt{n} + 1$ observations, both $|X_{(n/2 - \alpha \sqrt{n})}|
\le |X|_{(8\alpha\sqrt{n})}$ and $|X_{(n/2 + \alpha \sqrt{n})}| \le
|X|_{(8\alpha\sqrt{n})}$ hold. Therefore, on 
the event $E_1 \cap E_2$,
\begin{equation}
\label{eq:ialphacond}
|I_{\alpha}| \le 2|X|_{(8\alpha\sqrt{n})}~.
\end{equation}
Finally, we use Lemma \ref{lemmamedfirst} to control $|X|_{(8\alpha\sqrt{n})}$. By Markov's inequality and Lemma \ref{lemmamedfirst}, we have
\begin{eqnarray*}
\PROB\left(|X|_{(8\alpha\sqrt{n})} \ge t\right) 
& \le & \frac{\E|X|_{(8\alpha\sqrt{n})}^p}{t^p}  \\
& \le & t^{-p}\left(4\sqrt{2}\max\{p, \log(8\alpha\sqrt{n} + 1)\}\beta^{-1}\max\limits_{1 \le j \le 8\alpha\sqrt{n}}\frac{k + 1 - j}{\sum_{i = j}^n\sigma_i^{-1}}\right)^p~.
\end{eqnarray*}
Denote $\gamma =
4\sqrt{2}\beta^{-1}\max\limits_{1 \le j \le 8\alpha\sqrt{n}}\frac{k  + 1 - j}{\sum_{i = j}^n\sigma_i^{-1}}$. 
Provided that $\frac{t}{\gamma}e^{-1} \ge \log(8\alpha\sqrt{n} + 1)$, we may fix $p = \frac{t}{\gamma}e^{-1}$ and get
\[
\PROB\left(|X|_{(8\alpha\sqrt{n})} \ge t\right) \le \exp\left(-\frac{t}{e\gamma}\right)~.
\]
Fixing $t = \left(\log\frac{3}{\delta} \lor \log(8\alpha\sqrt{n} + 1)\right)e\gamma$ we have that, with probability at least $1 - \delta/3$,
\[
|X|_{(8\alpha\sqrt{n})} \le 4e\sqrt{2}\left(\log\frac{3}{\delta} \lor \log(8\alpha\sqrt{n} + 1)\right)\beta^{-1}\max\limits_{1 \le j \le 8\alpha\sqrt{n}}\frac{k + 1 - j}{\sum_{i = j}^n\sigma_i^{-1}}.
\]
Denote this event by $E_3$.  

Choosing $\alpha = \sqrt{2\log \frac
  {6}{\delta}}$ we have $\PROB(E_1) \ge 1 - \delta/3$. Since $\alpha \sqrt{n} \ge 2\alpha^2$, we have $\PROB(E_2) \ge 1 -  2\exp\left(-\alpha \sqrt{n} \right)\ge 1 -  2\exp\left(-4\log \frac{6}{\delta} \right) \ge 1 -
\delta/3$. Therefore, we have by the union bound, that $E_1 \cap E_2 \cap E_3$ is
of probability at least $1 - \delta$. On this event due to
\eqref{eq:ialphacond} we have
\begin{align*}
|I_{\alpha}| \le 8e\sqrt{2}\left(\log\frac{3}{\delta} \lor \log\left(8\alpha\sqrt{n}+ 1\right)\right)\beta^{-1}\max\limits_{1 \le j \le 8\sqrt{2n\log \frac {6}{\delta}}}\frac{k + 1 - j}{\sum_{i = j}^n\sigma_i^{-1}}~.
\end{align*}
The claim follows by observing that $k = 8\alpha\sqrt{n}\le n$ is equivalent to $128\log \frac {6}{\delta} \le n$.

\begin{corollary}
\label{cor:medianperformance}
Under the assumptions of Proposition \ref{prop:medianintterval} the median $X_{(n/2)}$ satisfies, with probability at least $1 - \delta$,
\[
|X_{(n/2)} - \mu| \le 8e\sqrt{2}\left(\log\frac{3}{\delta} \lor \log\left(n+ 1\right)\right)\beta^{-1}\max\limits_{1 \le j \le 8\sqrt{2n\log \frac {6}{\delta}}}\frac{8\sqrt{2n\log \frac {6}{\delta}} + 1 - j}{\sum_{i = j}^n\sigma_i^{-1}}.
\]
\end{corollary}
\begin{proof}
Indeed, with probability at least $1 - \delta$, both $\mu$ and
$X_{(n/2)}$ belong to $I_{\alpha}$ for $\alpha$ as in Proposition
\ref{prop:medianintterval}, and therefore, $|X_{(n/2)} - \mu| \le |I_\alpha|$. 
\end{proof}

\section{Modal interval estimator}
\label{sec:modalinterval}

The second component of our adaptive estimator is the simple and
natural estimator that looks for an interval of a given length
containing the maximum number of data points. This is the so-called
\emph{modal interval estimator} introduced by
\cite{chernoff1964estimation} for estimating the mode of a density function.
\cite{pensia2019estimating} also analyze this estimator though their 
bounds have some limitations for our purposes. We make a detailed comparison
in Section \ref{sec:comparisons} below.

In this section we work under the following assumptions.

\begin{assumption}{\mathbf{B}}\label{AssumptionB}
Let $X_1,\ldots,X_n$ be independent random variables such that $X_i$ has density
$(1/\sigma_i)\phi((x-\mu)/\sigma_i)$ where $\phi$ is some fixed density function,
$\mu$ is a location parameter and $\sigma_1 \le \cdots \le \sigma_n$ are positive scale parameters. Assume that

\begin{itemize}
\item[(i)]
$\int x\phi(x) dx =0$. This implies that $\EXP X_i = \mu$ for all $i \in [n]$.

\item[(ii)]
$\int x^2\phi(x) dx =1$. This implies that $\var(X_i) = \sigma_i^2$ for all $i \in [n]$.

\item[(iii)]
  Symmetry: $\phi(-x)=\phi(x)$ for all $x\in \R$.

\item[(iv)]
Unimodality: $\phi(x)$ is non-increasing for $x>0$ and non-decreasing for $x<0$.
\end{itemize}
\end{assumption}

An important example satisfying Assumption \ref{AssumptionB} is the Gaussian case, that is, when $\phi(x)=(1/\sqrt{2\pi}) e^{-x^2/2}$. 
However, in general, $\phi$ may have a heavy tail as long as the second moment
exists. We also do not need to assume that $\phi$ is
bounded. Introduce the notation 
\[
    \Phi(t) = \int_{-t}^t \phi(x) dx~.
\]
For $s>0$, denote the interval $A_s(x) = [x-s,x+s]$. 
Let
\[
     D_s(x) = \sum_{i=1}^n \IND_{X_i \in A_s(x)}
\]
be the number of points in the interval $A_s(x)$. 
Denoting $q_i(s) = \PROB\{X_i \in A_s(\mu)\}= \Phi(s/\sigma_i)$,
we have
\[
      \EXP D_s(\mu) = \sum_{i=1}^n q_i(s)~.
\]
Define the
\emph{modal interval estimator} which returns the center of the
densest interval of length $2s$. That is,

\medskip

\begin{tcolorbox}
\begin{equation}
\label{eq:muns}
    \wh{\mu}_{n,s}\in \argmax_{x\in \R} D_s(x)~.
\end{equation}
\end{tcolorbox}

For the modal interval estimator to work (in the sense that it
contains the common mean $\mu$), the length $s$ has to satisfy certain 
conditions. Such a sufficient condition is formulated in the following
definition that intuitively captures the fact that the densest interval
should contain $\mu$, even after accounting for random fluctuations. 
In Proposition \ref{thm:intervalest} below we prove the
condition of admissibility specified here is indeed sufficient.

\begin{definition}
\label{def:mdelta}
Fix the confidence $\delta >0$ and the interval length $s > 0$. Define
\[
m_s = \max\{m \in [n]: \sigma_m \le s\}~.
\]
We say that the length $s$ is \underline{admissible} if 
\[
m_s \ge \kappa\left( \sqrt{\EXP D_{s}(\mu)
        \log\frac{2n}{\delta}} +
      \log\frac{2n}{\delta} \right),
\]
where $\kappa > 0$ is a numerical constant.
 Finally, we set
\begin{equation}
\label{eq:mdelta}
    \ol{s}(\delta) = \inf\left\{s > 0: s\; \text{is admissible}\right\}.
\end{equation}
\end{definition}

\begin{remark}
The value of the constant $\kappa>0$ depends on a universal constant 
appearing in Lemma \ref{lem:vclemma} below. While it is possible to
extract a specific value, it is somewhat tedious and not crucial for
our arguments, so we prefer to keep it unspecified. 
All results below hold for all values of $\kappa \ge \kappa_0$ for
some constant $\kappa_0$. Changing the value only effects the
constants in the results below.
\end{remark}

\begin{remark}
\label{rem:newcond}
Observe that if the density is bounded, that is, if $\phi(0)$ is finite, we have $q_i(s) \le \min\{1, 2\phi(0)
s/\sigma_i\}$. Therefore, adjusting the constant $\kappa$, we may replace the admissibility criterion by the condition 
\[
m_s \ge \kappa\left( \sqrt{\left( \sum_{i=1}^n \min\left\{1, 2\phi(0)
\frac{s}{\sigma_i}\right\} \right)
        \log\frac{2n}{\delta}} +
      \log\frac{2n}{\delta} \right)~,
\]
Roughly speaking, whenever  $\phi(0)$ is finite one may think that $\ol{s}(\delta)$ is approximately equal to $\sigma_{m^*}$, where $m^*$ is the smallest integer satisfying
\[
m^* \gtrsim \sqrt{\sigma_{m^*}\left(\sum_{i = m^*}^n \frac{1}{\sigma_i}\right)\log \frac{n}{\delta}}~.
\]
\end{remark}

The main result of this section is the following bound.
\begin{proposition}
\label{thm:intervalest}
Let Assumption \ref{AssumptionB} be satisfied. Fix $\delta \in (0,
1)$. 
Then, with probability at least $1-\delta$, 
simultaneously for all admissible $s > 0$, it holds that
\[
     \left|\wh{\mu}_{n,s} - \mu\right| \le 4s ~.
\]
\end{proposition}

\begin{proof}
We start by showing a simple lower bound for $\Phi(1)= 2 \int_0^1
\phi(x) dx$. Fix any $t \ge 1$ and observe that by property $(iv)$ in
Assumption \ref{AssumptionB}, we have $t\Phi(1) \ge \Phi(t)$. At the
same time, by Chebyshev's inequality and property $(ii)$ we have $\Phi(t) > (1 - 1/t^2)$. Therefore,
\begin{equation}
\label{eq:philowerbound}
\Phi(1) \ge \sup\limits_{t \ge 1}\frac{1}{t}\left(1 - \frac{1}{t^2}\right) = \frac{2}{3\sqrt{3}}~.
\end{equation}
As the estimator is translation invariant, we may assume, without loss
of generality, that $\mu =0$. We show that, on
the one hand, with probability at least $1-\delta/2$, simultaneously for all admissible $s$, 
\begin{equation}
\label{eq:lower}
   \max_{x\in \R} D_s(x)\ge m_s \frac{3\Phi(1)}{4} + \sum_{i >m_s} q_i(s)~,
\end{equation}
and, on the other hand, with probability at least $1-\delta/2$,
\begin{equation}
\label{eq:upper}
   \max_{x\in \R: |x|\ge 4s} D_s(x)< m_s \frac{3\Phi(1)}{4} + \sum_{i >m_s} q_i(s)~.
\end{equation}
These two properties together imply the proposition.
First, we show (\ref{eq:lower}). Note that for $i \le m_s$ we have $\sigma_i \le s$ and $q_i(s) = \PROB\{X_{i} \in A_s(0)\} \ge \Phi(1)$, and therefore,
\begin{equation}
\label{eq:expdso}
\EXP D_s(0) = \sum_{i=1}^n q_i(s) \ge m_s \Phi(1) + \sum_{i >m_s} q_i(s)~.
\end{equation}
Observe that, since $s$ is admissible, we have 
\begin{align}
\label{eq:mslowerbound}
m_s\Phi(1) &\ge \Phi(1)\left(\kappa\sqrt{\E{D_{s}}(0)\log \frac{2n}{\delta}} + \kappa\log\frac{2n}{\delta}\right) \nonumber
\\
&\ge \frac{2}{3\sqrt{3}}\left(\kappa\sqrt{\E{D_{s}}(0)\log \frac{2n}{\delta}} + \kappa\log\frac{2n}{\delta}\right).
\end{align}
Denote $\kappa' = \frac{2}{3\sqrt{3}}\kappa$. Using \eqref{eq:mslowerbound}, we have 
\begin{align*}
   &\PROB\left\{ \exists s > 0: s\; \text{is admissible}, \max_{x\in \R} D_s(x) \le  m_s \frac{3\Phi(1)}{4} + \sum_{i >m_s} q_i
  \right\} \\
 &\le
   \PROB\left\{ \exists s > 0: s\; \text{is admissible}, D_{s}(0) \le  m_s \frac{3\Phi(1)}{4} + \sum_{i >m_s} q_i
        \right\}
        \\
        & \le   \PROB\left\{ \exists s > 0: s\; \text{is admissible}, D_{s}(0) \le  \EXP D_{s}(0) - m_s \frac{\Phi(1)}{4} \right\} 
        \\
 &\le    \PROB\left\{ \exists s > 0: s\; \text{is admissible}, \EXP D_s(0) - D_s(0) \ge  \frac{\kappa'}{4}\sqrt{\E{D_s}(0)\log \frac{2n}{\delta}}  + \frac{\kappa'}{4}\log\frac{2n}{\delta}\right\}.
\end{align*}
By Lemma \ref{lem:vclemma} in Appendix \ref{app:a}, since the {\sc vc} dimension of the family
intervals in $\R$ equals $2$, one may tune the value of $\kappa'$ such that the last probability is bounded by $\frac{\delta}{2}$. \par
We are now ready to analyze \eqref{eq:upper} for which it is enough to
show that
\begin{align}
\label{eq:probsecpart}
\PROB\Big\{ \exists s\ge 0  &  \ \text{and} \ 
|x|\ge (1+\sqrt{2/\Phi(1)})s: \nonumber \\
&  s \ \text{is admissible and } D_s(x) >   m_s \frac{3\Phi(1)}{4} + \sum_{i >m_s} q_i \Big\} \le \frac{\delta}{2}~.
\end{align}

Observe that by \eqref{eq:philowerbound} we have $1+\sqrt{2/\Phi(1)} \le 4$. Given $x$ such that $|x| > 4s > (1+\sqrt{2/\Phi(1)})s$, using the properties of the density $\phi$ together with $s \ge \sigma_{m_s}$ and Chebyshev's inequality, we have
\begin{eqnarray*}
\EXP D_s(x) & = & \sum_{i \le m_s}\PROB\{X_i \in A_{s}(x)\} + \sum_{i > m_s}\PROB\{X_i \in A_{s}(x)\}
\\
& \le & \sum_{i \le m_s}\PROB\{ |X_i| \ge \sqrt{2/\Phi(1)}s\} + \sum_{i > m_s}\PROB\{X_i \in A_{s}(x)\}
\\
& \le & m_s\frac{\Phi(1)}{2} + \sum_{i >m_s} q_i~.
\end{eqnarray*}
Using this inequality together with \eqref{eq:mslowerbound} and recalling that $\kappa' = \frac{2}{3\sqrt{3}}\kappa$, we have
\begin{align*}
    &\PROB\left\{  \exists s\ge 0 \;\text{and}\;  |x|\ge 4s: s\; \text{is admissible}, D_s(x) >
  m_s \frac{3\Phi(1)}{4} + \sum_{i >m_s} q_i \right\}
  \\&\le\PROB\left\{  \exists s\ge 0 \;\text{and}\;  |x|\ge 4s: s\; \text{is admissible}, D_s(x) > \E D_s(x) + 
  m_s \frac{\Phi(1)}{4}\right\}
  \\
  &\le\PROB\Bigl\{  \exists s\ge 0 \;\text{and}\;  |x|\ge 4s: s\; \text{is admissible}, 
 \\
 &\quad\quad\quad\quad\quad D_s(x) > \E D_s(x) + 
 \frac{\kappa'}{4}\sqrt{\E{D_s}(0)\log \frac{2n}{\delta}}  + \frac{\kappa'}{4}\log\frac{2n}{\delta}\Bigr\}~.
\end{align*}
Using $\E D_s(x) \le \E D_s(0)$, the last line is bounded by 
\begin{align*}
&\PROB\Bigg\{  \exists s\ge 0 \;\text{and}\;  |x|\ge 4s: s\; \text{is admissible}, 
\\&\quad \quad D_s(x) > \E D_s(x) + 
 \frac{\kappa'}{4}\sqrt{\E{D_s}(x)\log \frac{2n}{\delta}}  + \frac{\kappa'}{4}\log\frac{2n}{\delta}\Bigg\}~.
\end{align*}
Finally, the last expression and Lemma \ref{lem:vclemma}, which holds simultaneously for all $x$ and $s$, implies \eqref{eq:probsecpart} by adjusting the constant $\kappa$ (and thus $\kappa'$). The proof is complete.
\end{proof}

\section{An adaptive estimator: combining the median and the modal Interval}
\label{sec:adaptiveestimator}

Proposition \ref{thm:intervalest} shows that, as long as $s$ is an
\emph{admissible} value, the modal interval estimator has an error
bounded by $4s$. Hence, to optimize the bound, one should choose $s$
to be the smallest possible admissible value, that is,
$\ol{s}(\delta)$ introduced in Definition \ref{def:mdelta}. However,
the value of $\ol{s}(\delta)$ depends on the values
$\sigma_1,\ldots, \sigma_n$ and therefore one doesn't have access to 
$\ol{s}(\delta)$ unless the standard deviations are known (up to a
permutation), a typically unrealistic requirement. In this section we
introduce an adaptive estimator that is able to find an approximate
value of $\ol{s}(\delta)$ based only on the available data
$X_1,\ldots,X_n$.
Furthermore, the adaptive estimator combines the $\alpha$-median
interval estimator with the modal interval estimator and achieves an
error that is at least as good as the best of the median and the
optimal modal interval estimator, up to a constant factor.

The key to making the estimator adaptive is an empirical criterion,
based on which one can reject values of $s$ that are not admissible. 
Once one has such a criterion, standard techniques of adaptive
estimation may be applied (such as Lepski's method
\citep{lepskii1992asymptotically}). 

\medskip

\begin{tcolorbox}
Fix $\delta > 0$ and $s > 0$. Let $\eta, \xi > 0$ be numerical constants specified in the proof. Based on $X_1, \ldots, X_n$ let $\wh{\mu}_{n,s}$ be any maximizer of $D_s(x)$ defined by \eqref{eq:muns}.
\begin{itemize}
    \item We ACCEPT the interval $A_{s}(\wh{\mu}_{n,s})$ if $D_s(\wh{\mu}_{n,s}) \ge \xi\log \frac{2n}{\delta}$ and 
    \[
    \max\limits_{x \in \mathbb{R}, |x - \wh{\mu}_{n,s}| \ge 8s} D_{s}(x) \le D_{s}(\wh{\mu}_{n,s}) - \eta\left(\sqrt{{D_s}(\wh{\mu}_{n,s})\log \frac{2n}{\delta}} + \log \frac{2n}{\delta}\right)~.
    \]
    \item Otherwise, we REJECT this interval.
\end{itemize}
\end{tcolorbox}

\medskip

\begin{remark}
Since we only consider $D_s(\wh{\mu}_{n,s}) \ge \xi\log \frac{2n}{\delta}$ we may instead consider a criterion of the form 
\[
\max\limits_{x \in \mathbb{R}, |x - \wh{\mu}_{n,s}| \ge 8s} D_{s}(x) \le D_{s}(\wh{\mu}_{n,s}) - \eta^{\prime}\sqrt{{D_s}(\wh{\mu}_{n,s})\log \frac{2n}{\delta}}~,
\]
for some $\eta^{\prime} > 0$. However, the choice above makes the proof more transparent.
\end{remark}

This criterion satisfies the following relation.
\begin{proposition}
\label{prop:accept}
With probability at least $1 - \delta$, simultaneously for all $s>0$,
no interval with $|\wh{\mu}_{n,s} - \mu| > 8s$ is accepted and 
every admissible interval is accepted.
\end{proposition}

\begin{proof}
Recall that, without the loss of generality, we set $\mu = 0$. From now on we work on the event $E_1$ where the inequalities of Lemma \ref{lem:vclemma} hold.
We begin by proving that any admissible $s > 0$ is accepted with high
probability. We have shown in the proof of Proposition
\ref{thm:intervalest} that one the event $E_1$,
for all admissible values of $s$,
\[
|\wh{\mu}_{n,s}| \le 4s~.
\]
Therefore, on this event any $x \in \mathbb{R}$ such that $|x - \wh{\mu}_{n,s}| \ge 8s$ satisfies $|x| \ge 4s$. Also, by the argument in the proof of Proposition
\ref{thm:intervalest} and Lemma \ref{lem:vclemma} we have for all $|x| \ge 4s$, 
\begin{align}
\label{eq:thelowerpart}
D_{s}(x) &\le m_s\frac{\Phi(1)}{2} + \sum_{i >m_s} q_i + c_1\left(\sqrt{\E{D_s}(x)\log \frac{2n}{\delta}}  + \log\frac{2n}{\delta}\right) \nonumber
\\
&\le m_s\frac{\Phi(1)}{2} + \sum_{i >m_s} q_i + c_1\left(\sqrt{\E{D_s}(0)\log \frac{2n}{\delta}}  + \log\frac{2n}{\delta}\right),
\end{align}
where $c_1 > 0$ is a numerical constant.

Observe that the function $y \mapsto y -  \eta\sqrt{y\log
  \frac{2n}{\delta}}$ is increasing whenever $y > \eta^2\log \frac{2n}{\delta}/4$. 
Thus, $D_{s}(0) \ge \eta^2\log \frac{2n}{\delta}/4$ implies
\begin{equation}
\label{eq:relattoinineq}
D_s(\wh{\mu}_{n,s}) -  \eta\left(\sqrt{{D_s}(\wh{\mu}_{n,s})\log \frac{2n}{\delta}}  + \log\frac{2n}{\delta}\right) \ge D_s(0) -  \eta\left(\sqrt{{D_s}(0)\log \frac{2n}{\delta}}  + \log\frac{2n}{\delta}\right)~.
\end{equation}
Observe also that the line \eqref{eq:usefuleq} in the proof of Lemma \ref{lem:vclemma} implies that on the event $E_1$ it holds simultaneously for all $x$ that 
\begin{equation}
\label{eq:ratiotype}
D_{s}(x) \le 2\E D_s(x) +  c_2\log\frac{2n}{\delta}~, 
\end{equation}
where $c_2 > 0$ is a numerical constant.
By \eqref{eq:lower}, on the same event,
$
D_s(0) \ge m_s \frac{3\Phi(1)}{4} + \sum_{i >m_s} q_i(s)~,
$ 
and therefore, using the admissibility of $s$, the inequality \eqref{eq:relattoinineq} implies
\begin{align*}
&D_s(\wh{\mu}_{n,s}) -  \eta\left(\sqrt{{D_s}(\wh{\mu}_{n,s})\log \frac{2n}{\delta}}  + \log\frac{2n}{\delta}\right) 
\\
&\ge m_s \frac{\Phi(1)}{2} + m_s \frac{\Phi(1)}{4} + \sum_{i >m_s} q_i(s) - \eta\left(\sqrt{{D_s}(0)\log \frac{2n}{\delta}}  + \log\frac{2n}{\delta}\right)
\\
&
\ge  m_s \frac{\Phi(1)}{2} + \frac{\kappa'}{4}\sqrt{\E{D_s}(0)\log \frac{2n}{\delta}}  + \frac{\kappa'}{4}\log\frac{2n}{\delta}
\\
&\quad\quad + \sum_{i >m_s} q_i(s) - \eta\left(\sqrt{{D_s}(0)\log \frac{2n}{\delta}}  + \log\frac{2n}{\delta}\right)
\\
&\ge  m_s \frac{\Phi(1)}{2} + \frac{\kappa'}{4}\sqrt{\E{D_s}(0)\log \frac{2n}{\delta}}  + \frac{\kappa'}{4}\log\frac{2n}{\delta}
\\
&\quad\quad + \sum_{i >m_s} q_i(s) - \eta\left(\sqrt{2\E{D_s}(0)\log \frac{2n}{\delta}}  + (1 + c_2)\log\frac{2n}{\delta}\right)
~,
\end{align*}
where $\kappa^{\prime}$ is defined in the proof of Proposition \ref{thm:intervalest} and in the last line we used \eqref{eq:ratiotype} together with $\sqrt{a + b} \le \sqrt{a} + \sqrt{b}$ for $a, b \ge 0$.
Comparing this with \eqref{eq:thelowerpart} and choosing a
sufficiently large value of $\kappa$ in Definition \ref{def:mdelta}, we prove that 
admissible intervals
are accepted with high probability. 
It is only left to check that $D_s(\wh{\mu}_{n,s}) \ge \xi\log \frac{2n}{\delta}/4$ and that, given that the constant $\xi$ is properly adjusted, our additional acceptance assumption $D_s(\wh{\mu}_{n,s}) \ge \xi\log \frac{2n}{\delta}/4$ implies, with high probability, that $D_{s}(0) \ge \eta^2\log \frac{2n}{\delta}/4$ which was used in \eqref{eq:relattoinineq}. This computation follows immediately from Lemma \ref{lem:vclemma} and the fact that $\E D_s(x)$ is maximized at $x = 0$.

It remains to prove that our empirical criterion can never accept the interval with its center $\wh{\mu}_{n,s}$ satisfying $|\wh{\mu}_{n,s} | > 8s$. To do so we observe that if $|\wh{\mu}_{n,s}| > 8s$ then the interval $A_{8s}(\wh{\mu}_{n,s})$ does not contain $\mu = 0$ and in the acceptance criterion we should compare with $D_s(0)$.
Assuming that $\eta > 2\kappa_2$, where $\kappa_2$ is defined in Lemma \ref{lem:vclemma} and using that $\E D_s(x)$ is maximized at $0$, we have on the event where the inequalities of Lemma \ref{lem:vclemma} hold
\begin{align*}
&D_{s}(\wh{\mu}_{n,s}) - \eta\sqrt{{D_{s}}(\wh{\mu}_{n,s})\log \frac{2n}{\delta}} - \eta\log\frac{2n}{\delta}  
\\
&< \E \left[D_s(\wh{\mu}_{n,s})|X_1,\ldots,X_n\right] - \kappa_2\sqrt{{D_{s}}(\wh{\mu}_{n,s})\log \frac{2n}{\delta}} - \kappa_2\log\frac{2n}{\delta}
\\
&\le \E \left[D_s(\wh{\mu}_{n,s})|X_1,\ldots,X_n\right] - \kappa_2\sqrt{{D_{s}}(0)\log \frac{2n}{\delta}} - \kappa_2\log\frac{2n}{\delta}
\\
&\le \E D_s(0) - \kappa_2\sqrt{{D_{s}}(0)\log \frac{2n}{\delta}} -\kappa_2\log\frac{2n}{\delta}
\\
&\le D_s(0)~.
\end{align*}
Therefore, $D_{s}(\wh{\mu}_{n,s}) - \eta\sqrt{{D_{s}}(\wh{\mu}_{n,s})\log \frac{2n}{\delta}} - \eta\log\frac{2n}{\delta}  < \max\limits_{x \in \mathbb{R}, |x - \wh{\mu}_{n,s}| \ge 8s} D_{s}(x)$
, which implies that the interval $A_s(\wh{\mu}_{n,s})$ is rejected.
\end{proof}

\subsubsection*{The adaptive estimator}

We are now ready to define a fully adaptive estimator that achieves a
performance that is at least as good -- up to a constant factor -- as the best of our bounds for the
median (Proposition \ref{prop:medianintterval}) and the 
modal interval with optimally chosen length (Proposition \ref{thm:intervalest}).

\medskip

\begin{tcolorbox}
We observe a sample of independent random variables $X_{1}, \ldots, X_n$. Fix the desired confidence level $\delta \in (0, 1)$. We output the estimator $\wh{\mu}$ defined as follows:
\begin{itemize}
\item Fix $\alpha = \sqrt{2\log \frac {6}{\delta}}$ and compute the $\alpha$-median interval $I_{\alpha}$.
\item Let $\wh{\mu}$ be the midpoint of the interval
\begin{equation}
\label{eq:acceptanceset}
\left(\bigcap_{\substack{0 \le s \le |I_{\alpha}| \\ A_{s}(\wh{\mu}_{n,s})\; \text{is ACCEPTed}}}A_{8 s}(\wh{\mu}_{n,s})\right) \cap I_{\alpha},
\end{equation}
where $\wh{\mu}_{n,s}$ is defined by \eqref{eq:muns} and let $\wh\mu$
be the midpoint of the interval $I_{\alpha}$ if the set \eqref{eq:acceptanceset} is empty.
\item Return $\wh \mu$.
\end{itemize}
\end{tcolorbox}

\medskip

\begin{remark}
  In practice there is no need to search through all $s > 0$ in
  \eqref{eq:acceptanceset}. One may discretize and consider only
  $s_i = 2^{-i}|I_\alpha|$ for integers $i \ge 0$. Also, due to Lemma
  \ref{lem:vclemma} we may essentially replace $\E D_{s}(x)$ by
  $D_{s}(x)$ in the steps of the proof where admissibility is
  used. That is, one may instead consider the \emph{random}
  admissibility condition of the form
\[
m_s \gtrsim \sqrt{D_{s}(\mu)
        \log\frac{2n}{\delta}} +
      \log\frac{2n}{\delta}~.
\]
Due to the discrete nature of the sample, only a finite number of
values $D_{s}(\mu)$ is possible and in the set
\eqref{eq:acceptanceset} one may consider only at most $n \choose 2$
values of $s$ that correspond to the distances between pairs of
points. For the sake of brevity we omit the straightforward details of
the analysis of the discretized estimator and focus on the estimator
defined above.
\end{remark}

\begin{theorem}
\label{thm:adaptiveestimator}
Let Assumptions \ref{AssumptionA} and \ref{AssumptionB} be satisfied. Fix $\delta \in (0, 1/2)$ such that $128\log \frac {6}{\delta} \le n$. There is a numerical constant $c_1 > 0$ such that, with probability at least $1 - 2\delta$, the estimator $\wh{\mu}$ defined above satisfies 
\[
|\wh{\mu} - \mu| \le c_1\min\left\{ \ol{s}(\delta), \beta^{-1}\log\left(\frac{n}{\delta}\right)\max\limits_{1 \le j \le 8\alpha\sqrt{n}}\frac{8\alpha\sqrt{n} + 1 - j}{\sum_{i = j}^n\sigma_i^{-1}}\right\}~,
\]
where $\ol{s}(\delta)$ is given by Definition \ref{def:mdelta}.
\end{theorem}

\begin{proof}
Recalling that $|I_{\alpha}|$ is a random variable, consider the event $E_1$,
\[
\ol{s}(\delta) \le |I_{\alpha}|~.
\]
On the complementary event $\ol{E}_1$ we have that the solution based
on the $\alpha$-median interval is better than what one can get with 
the modal interval estimator. In particular, since $\wh{\mu}$ always
returns a point in $I_{\alpha}$, 
the proof is complete by Proposition \ref{prop:medianintterval}. 

Otherwise, we focus on the event $E_{1}$. Let $E_2$ 
be the event that every accepted interval $A_{s}(\wh{\mu}_{n,s})$ satisfies $\mu \in A_{8 s}(\wh{\mu}_{n,s})$.
By Proposition \ref{prop:accept}, it holds that $\PROB\{E_2\} \ge 1 -
\delta$. Therefore, on $E_2$, we have either 
\begin{equation}
\label{eq:muininterval}
\mu \in \bigcap_{\substack{0 \le s \le |I_{\alpha}| \\ A_{s}(\wh{\mu}_{n,s})\; \text{is ACCEPTed}}}A_{8 s}(\wh{\mu}_{n,s})~,
\end{equation}
or there are no accepted intervals in this range. The latter cannot be true on the event $E_1 \cap E_2$ since $\ol{s}(\delta) \le |I_{\alpha}|$ and $A_{\ol{s}(\delta)}(\wh{\mu}_{n,\ol{s}(\delta)})$ is accepted. Therefore, the intersection of intervals in \eqref{eq:muininterval} is non-empty and its length is bounded by $16\ol{s}(\delta)$. Thus, on the event $E_1 \cap E_2$ we have $|\wh{\mu} - \mu| \le 16\ol{s}(\delta)$. The claim follows by the union bound.
\end{proof}

\section{Examples and a comparison with existing results}
\label{sec:comparisons}

To demonstrate the meaning of the derived performance bounds, in this section we
discuss several natural examples and compare our results with existing
general bounds. As already mentioned, our adaptive estimator is
closely related to the estimator of 
\cite{chierichetti2014learning} and to the hybrid estimator of 
\cite{pensia2019estimating}. 
Apart from the full adaptivity of our estimator, 
let us emphasize some technical differences with the latter work (which generalizes the results in \citet{chierichetti2014learning}):
\begin{itemize}
\item Even though the results in \citep{pensia2019estimating} work under milder assumptions, their bounds depend on the distribution through the quantity $r_k$ which should be ``manually'' computed in each particular case. In contrast, our results require that Assumptions \ref{AssumptionA} and \ref{AssumptionB} hold, but because of this the resulting bound depends explicitly on the standard deviations $\sigma_{1}, \ldots, \sigma_n$.
\item Our analysis of the modal interval estimator is sharper. In
  particular, while by \citet[Theorem 1]{pensia2019estimating} the
  modal interval estimator can never choose a center that has on
  average less than $\frac{1}{2}\E D_{s}(\mu)$ observations, our
  analysis uses the sharper property that the modal interval estimator
  never chooses a center that has, on the average, less than $\E  D_{s}(\mu) - c\sqrt{\E D_{s}(\mu)}$ observations, for some $c > 0$ up to logarithmic factors. 
\end{itemize}

\subsection{Examples}

 Most of our examples appear in
\citep{pensia2019estimating} and 
\citep{yuan2020learning}. 
We show that our bounds written in terms of $\sigma_{1}, \ldots, \sigma_n$
 are not worse than any of the previous bounds depending on some more
 involved distribution dependent quantities, often achieved by
 non-adaptive estimators. In all
examples we only consider the Gaussian case, that is, we assume
$X_i \sim \mathcal{N}(\mu, \sigma_i^2)$. Also, for the sake of
presentation we fix the allowed probability of error to be $\delta =
\frac{1}{n}$. 

\begin{example}(Equal variances.)
In the simplest case we have $\sigma_i = \sigma$ for $i \in [n]$.
In this example the median interval alone recovers the optimal error
rate $\wt{O}\left(\frac{\sigma}{\sqrt{n}}\right)$. 
Therefore, our adaptive algorithm mimics the optimal behavior of the sample mean in the i.i.d.\ scenario.
\end{example}

\begin{example}(Two variances.)
Consider the case where $\sigma_i = \sigma$ for $i \in [m]$ and 
$\sigma_i = \sigma^{\prime}> \sigma$ for $i \in [n]\setminus [m]$.

There are different cases and we consider the most interesting regimes. First, if $m \gtrsim \sqrt{n\log n}$ the median gives the rate $\wt{O}\left(\frac{\sqrt{n}}{m\sigma^{-1} + (n -  m)(\sigma^{\prime})^{-1}}\right)$ and the interval algorithm can always guarantee the error $O(\sigma)$ since the interval of length $\sigma$ is admissible. Next we consider $m \lesssim \sqrt{n\log n}$. The median gives the rate $\wt{O}\left(\frac{\sigma^{\prime}}{\sqrt{n}}\right)$ in this regime and the interval of length $O(\sigma)$ is admissible if $m \gtrsim \sqrt{\frac{\sigma}{\sigma^{\prime}}n\log n} + \log n$. In particular, an application of Theorem \ref{thm:adaptiveestimator} and shows that, with probability at least $1 - \frac{1}{n}$,
\[
|\wh \mu - \mu| = \begin{cases} \wt{O}\left(\frac{\sqrt{n}}{m\sigma^{-1} + (n -  m)(\sigma^{\prime})^{-1}}\right), & \mbox{if } m \gtrsim \sqrt{n \log n};
\\ \wt{O}\left(\sigma \wedge \frac{\sigma^{\prime}}{\sqrt{n}}\right), & \mbox{if } \sqrt{\frac{\sigma}{\sigma^{\prime}}n\log n} + \log n \lesssim m \lesssim \sqrt{n \log n}.
\end{cases}
\]
\end{example}

\begin{example}($\alpha$-mixture distributions.)
This is a particular case of the example of two variances above, with $m = c\lfloor \log n \rfloor$, for some $c > 0$; $\sigma = 1$ and $\sigma^{\prime} = n^{\alpha}$ for some $\alpha > 0$. 
This example was thoroughly studied in
\citep{pensia2019estimating}. When $\alpha < 1$ the analysis of the sample median in the
example above gives, with probability at least $1 - \frac{1}{n}$,
\[
|\wh \mu - \mu| = \wt{O}(n^{\alpha - 1/2})~,
\]
otherwise, for $\alpha \ge 1$ provided that $c$ is a large enough numerical constant we have
\[
|\wh \mu - \mu| = O(1)~.
\]
Therefore, our algorithm recovers the best known rates in \cite[Table
1 and Proposition 5]{pensia2019estimating}, in an adaptive manner.
\end{example}

\begin{example}(Quadratic variances.)
In this setup we assume that for some constant $c >0$, $\sigma_i^2 = c^2i^2$. 
In this case,  an interval of length $s = cj$ is admissible if
\[
j \gtrsim \sqrt{\sum_{i = j}^{n}\frac{j}{i}\log n} + \log n~.
\]
Using $\sum_{i = j}^{n}\frac{j}{i} \lesssim j\log \frac{n}{j}$, we see
that an interval of length proportional to $\log n$ is admissible. 
A simple computation shows that the median interval can produce an
error $\wt{O}(\sqrt{n})$ (see also \cite[Proposition
4]{pensia2019estimating}). 
Finally, an application of Theorem \ref{thm:adaptiveestimator} gives, with probability at least $1 - \frac{1}{n}$,
\[
|\wh \mu - \mu| =  O(\log n)~.
\]
This improves upon the bound of \citet[Table 1 and Proposition
5]{pensia2019estimating} 
where for the same model an arbitrarily  small polynomial error is established.
\end{example}

\begin{example}(The subset-of-signals model.)
In this setup the only assumption is that, for some $m<n$, at least $m$ out of $n$
variances are less or equal to one. In other words, $\sigma_m \le 1$.
The subset-of-signals model was studied by 
\cite{yuan2020learning}. 
The authors prove that if $m \gtrsim \sqrt{n \log n}$, then there is an estimator $\wt{\mu}$ based on iterative truncations (first studied in \citep{yuan2020learning_a}) such that, with probability at least $1 - 1/n$,
\[
|\wt{\mu} - \mu| \lesssim \frac{\sqrt{n \log n}}{m}~.
\]
Assuming that $m \gtrsim \sqrt{n\log n}$, we have by Proposition
\ref{prop:medianintterval} and Theorem \ref{thm:adaptiveestimator},
that, with probability at least $1 - \frac{1}{n}$,
\[
|I_\alpha|\lesssim \frac{\sqrt{n}(\log n)^{3/2}}{m} \quad \text{and thus,} \quad |\wh \mu - \mu| =  \wt{O}\left(\frac{\sqrt{n}}{m}\right)~.
\]
This shows that the sample median (and hence our general adaptive
estimator) performs as well as the algorithm of
\citet{yuan2020learning}, up to a logarithmic factor. The advantage of
the median is that its complexity is linear
in the number of observations \citep{blum1973time} whereas the
iterative truncation algorithm is more complex. Moreover,
the iterative truncation algorithm of \citet{yuan2020learning} depends
on some parameters of the problem as well as on an initialization. We
additionally remark that according to \citeauthor{yuan2020learning} the
hybrid estimator of \citet{pensia2019estimating} also recovers the
rate $\wt{O}\left(\frac{\sqrt{n}}{m}\right)$ in the subset-of-signals
model.
\end{example}

\subsection{A comparison with some general bounds}

Finally, we compare our results with several recent general
bounds. Our main conclusion is that, apart from the logarithmic
factors and at least in the case of Gaussian data, our adaptive estimator
performs at least as well as the best known guarantees in the
literature. We emphasize again that our estimator does not depend on
any parameters of the problem whereas the best known algorithms
require some kind of parameter tuning.

\subsubsection*{The result of Xia on the median of Gaussians}

\citet{xia2019non} analyzed the sample median of independent, not
necessarily identically distributed random variables with the same median. For the sake of an
easier comparison, we only consider here the case of normal random variables.
The following result appears in \citep[Corollary 6]{xia2019non}.

\begin{proposition}
\label{th:xia19}
Consider independent $X_{1}, \ldots, X_i$ such that $X_i \sim \mathcal{N}(\mu, \sigma_i^2)$. Assume that $\delta \in (0, 1)$ satisfies  
\begin{equation}
\label{eq:condition}
\frac{\sqrt{n\log \frac{1}{\delta}}}{\sum_{i = 1}^n \sigma_{i}^{-1}} \le \frac{7\sqrt{2}\sigma_1}{10}~.
\end{equation}
Then, with probability at least $1 - \delta$,
\[
|X_{(n/2)} - \mu| \le \frac{\frac{10}{7}\sqrt{2n\log \frac{1}{\delta}}}{\sum_{i = 1}^n \sigma_{i}^{-1}}~.
\]
\end{proposition}

At first glance the result of Proposition \ref{th:xia19} looks
stronger than what is given by Corollary \ref{cor:medianperformance}
in the special case of Gaussians. Indeed, it does not have the
$\log n$ factor and has a better dependence on
$\log \frac{1}{\delta}$. The main difference comes from the assumption
\eqref{eq:condition} which is more restrictive than the only
assumption $128\log \frac {6}{\delta} \le n$ of Corollary
\ref{cor:medianperformance}. Indeed, in
the most favourable case  when $\sigma_{i} = \sigma$ for $i \in [n]$,
the condition \eqref{eq:condition} implies
$\log \frac{1}{\delta} \le \left(\frac{7\sqrt{2}}{10}\right)^2n$ which
coincides with our assumption up to absolute constants. However, for
small $\sigma_1$ the assumption \eqref{eq:condition} requires
$\delta \to 1$ whereas our bound is not sensitive to the approximately
$\sqrt{n\log \frac{1}{\delta}}$ smallest variances.  The following
result shows that the condition \eqref{eq:condition} simplifies the
bound of Proposition \ref{prop:medianintterval} making it almost the
same as the result of Proposition \ref{th:xia19}, up to logarithmic
factors.

\begin{corollary}
Fix $\delta \in (0, 1)$ such that $128\log \frac {6}{\delta} \le n$ and set $\alpha = \sqrt{2\log \frac {6}{\delta}}$. Assume that there is $0 < c < 1$ such that
\[
\frac{8\sqrt{2n\log \frac {6}{\delta}}}{\sum_{i = 1}^n \sigma_{i}^{-1}} \le c\sigma_{1}~.
\]
Under the assumptions of Proposition \ref{prop:medianintterval} we have, with probability at least $1-\delta$,
\[
|I_{\alpha}| \le 64e\sqrt{2}\left(\log\frac{3}{\delta} \lor \log\left(n+ 1\right)\right)\beta^{-1}\frac{\sqrt{2n\log \frac {6}{\delta}}}{(1 - c)\sum_{i = 1}^n\sigma_i^{-1}}~.
\]
\end{corollary}
\begin{proof}
The proof is based on elementary comparisons. Fix $j \le 8\alpha\sqrt{n}$. Then
\begin{align*}
    \sum_{i = j}^n\sigma_i^{-1} &\ge \sum_{i = 8\alpha\sqrt{n} + 1}^n\sigma_i^{-1}= \sum_{i = 1}^n\sigma_i^{-1} - \sum_{i = 1}^{8\alpha\sqrt{n}}\sigma_i^{-1}
    \\
    &\ge 8\alpha\sqrt{n}\; \sigma_1^{-1}c^{-1} - \sum_{i = 1}^{8\alpha\sqrt{n}}\sigma_{i}^{-1}\ge (c^{-1} - 1)\sum_{i = 1}^{8\alpha\sqrt{n}}\sigma_{i}^{-1}.
\end{align*}
This implies
\[
c^{-1}\sum_{i = j}^n\sigma_i^{-1} \ge (c^{-1} - 1)\sum_{i = j}^n\sigma_i^{-1} + (c^{-1} - 1)\sum_{i = 1}^{8\alpha\sqrt{n}}\sigma_{i}^{-1} \ge (c^{-1} - 1)\sum_{i = 1}^{n}\sigma_{i}^{-1}.
\]
Combining these inequalities, we obtain
\[
\max_{1 \le j \le 8\alpha\sqrt{n}}\frac{8\alpha\sqrt{n} - j + 1}{\sum_{i = j}^n\sigma_i^{-1}} \le \frac{1}{1 - c}\max_{1 \le j \le 8\alpha\sqrt{n}}\frac{8\alpha\sqrt{n} - j + 1}{\sum_{i = 1}^n\sigma_i^{-1}} = \frac{1}{1 - c}\frac{8\alpha\sqrt{n}}{\sum_{i = 1}^n\sigma_i^{-1}}~.
\]
The result follows.
\end{proof}

\subsubsection*{The bound of \citeauthor*{chierichetti2014learning}.}

\citet[Theorem 4.1]{chierichetti2014learning} introduce an estimator $\wt{\mu}$ such that for $X_i \sim \mathcal{N}(\mu, \sigma_i)$, with probability at least $1 - \frac{1}{n}$,
\begin{equation}
\label{eq:siloviosbound}
|\wt{\mu} - \mu| = \wt{O}(\sigma_{\log n}\sqrt{n})~.
\end{equation}
The hybrid estimator of \citet[Theorem 6 and Lemma
1$(v)$]{pensia2019estimating} satisfies a similar performance bound 
if the parameters are chosen in a specific way. The next result shows
that the adaptive estimator introduced in this note achieves this
bound without any additional parameter tuning. Moreover, the result follows from our general bounds written in terms of $\sigma_{1}, \ldots, \sigma_n$.

\begin{proposition}
\label{prop:silviosbound}
Let Assumptions \ref{AssumptionA} and \ref{AssumptionB} hold and
assume that $\log n$ is integer.
There is a constant $c = c(\beta, \phi(0)) > 1$ such that the adaptive estimator
of Theorem \ref{thm:adaptiveestimator} satisfies for large enough $n$ that, with
probability at least $1 - \frac{1}{n}$,
\[
|\wh{\mu} - \mu| \le c\left(\sigma_{c\log n}\sqrt{n} \log^{3/2} n\right)~.
\]
\end{proposition}

The proof is based on some elementary but tedious computations, see
Appendix B.

\section{Concluding remarks}
\label{sec:concremark}

In this note we construct a fully adaptive estimator for the common
mean of independent, not necessarily identically distributed random
variables and provide performance guarantees that hold under certain
assumptions for the underlying distribution. The key assumptions
are that the distributions are symmetric around the mean and the
underlying densities are unimodal. However, even in the simplest case
of normal random variables, the problem is not fully understood. 
In particular, as far as we know, no general nontrivial lower bounds are
available. It is not difficult to prove that no estimator can have 
an expected error smaller than that of the maximum likelihood
estimator that ``knows'' the variance of each sample point, that is,
$\left(\sum_{i=1}^n \sigma_i^{-2}\right)^{-1/2}$. In the absence of knowledge of the $\sigma_i$, the
problem becomes significantly harder. It remains an interesting challenge 
to prove general lower bounds that are much larger than the trivial bound
$\left(\sum_{i=1}^n \sigma_i^{-2}\right)^{-1/2}$.
In fact, we think that, up to logarithmic factors, the upper
bound of Theorem \ref{thm:adaptiveestimator} is essentially tight for
most interesting values of the parameters. However, the full picture
is surely more complex. For example, in some particular ranges of the parameters
it is easy to improve on Theorem \ref{thm:adaptiveestimator}. To
illustrate such an example, consider the case of two variances 
discussed in Section \ref{sec:comparisons}, that is,
 when $\sigma_i = \sigma$ for $i \in [m]$ and 
$\sigma_i = \sigma^{\prime}> \sigma$ for $i \in [n]\setminus [m]$.
Suppose that $\sigma \sqrt{\log m} \ll \sigma'/n$. In this case, with
high probability, the modal interval of length $s=3\sigma \sqrt{\log
  m}$ contains all of $X_1,\ldots,X_m$ but none of
$X_{m+1},\ldots,X_n$. In this case, instead of outputting the center
of the modal interval, by averaging the points falling in it, one
obtains an error of the order $O(\sigma/\sqrt{m})$, as opposed to
$O(\sigma)$ guaranteed by Theorem \ref{thm:adaptiveestimator}
in this case.

Even our analysis of the sample median leaves room for improvement. 
In particular, we think that part (iii) of Assumption
\ref{AssumptionA} may be weakened. While it is obviously necessary to assume
that the density of the $X_i$ are bounded away from zero near the
mean (consider the case of independent Rademacher random signs in the i.i.d.\ case), the exponential tail condition implied by this assumption seems
unnecessary. Indeed, Corollary 12 in \citep{xia2019non} deals
with the heavy-tailed Cauchy
distribution. 

Another interesting challenge is to gain an understanding of more
general cases when $X_1,\ldots,X_n$ are independent, they have the
same mean, but their distribution may not be symmetric or unimodal.

Finally, we mention that the model studied in this note 
is closely related to the model of heteroscedastic linear regression
with fixed design.
 In this model it is assumed that one observes, for $i \in [n]$,
\[
Y_i = \inr{x_i, \beta} + \xi_i~,
\]
where $\beta \in \mathbb{R}^d$ is the target parameter, $x_i \in \mathbb{R}^d$ are deterministic design vectors, and $\xi_i \sim \mathcal{N}(0, \sigma_i)$ are independent noise variables. In order to provide some reasonable guarantees for this model, one usually makes some additional assumptions. In a classical model (see, for instance, \citep{fuller1978estimation}) it is assumed that the values of $\sigma_i$ are arbitrary,
but there are enough repetitions of each observation available so that
one can estimate the values of $\sigma_i$. 
Once the values of $\sigma_i$ and their assignments to the
observations are (almost) known, one may use the weighted mean
described in the introduction
which achieves (almost) optimal performance. 
Another line or research which can be attributed, among other papers,
to the early work of \citet{carroll1982robust}, is where some
additional assumptions on $\sigma_i$ are made. For example, they are
increasing according to some law. 
Our model can be seen as a particular case of heteroscedastic linear
regression in dimension one, where we additionally assume that the
design $x_i$ is the same for all $i$. However, these simplifications
are compensated by the fact that we make neither the assumption on the
repeated observations nor the assumption that the $\sigma_i$ follow a
particular functional form. 
Finally, our estimators are invariant to the permutation of the elements of the sample and thus cannot exploit the monotonicity of the standard deviations.

\appendix
\section{Ratio-type {\sc vc} bounds for non-identically distributed entries}
\label{app:a}
In this section we provide high probability ratio-type {\sc vc} bounds
(originally due to \cite*[Theorem 12.2]{Vapnik74}) for independent but
non-identically distributed random variables. A 
bound of a similar type was proved in \cite[Lemma 2]{pensia2019estimating} though
their result is not sufficient for our purposes \footnote{In particular, our result
  covers some values of their parameter $t$ that are not allowed in \cite[Lemma
  2]{pensia2019estimating}.}. Consider a set $\mathcal{F}$ of
$\{0,1\}$-valued functions defined on a domain $\mathcal{X}$ such
with {\sc vc}-dimension equal to $d$. Recall that the {\sc vc}
dimension that is the largest
integer $d$ such that there are $x_1, \ldots, x_d \in \mathcal{X}$
satisfying
$\bigl|\{(f(x_1), \ldots, f(x_d)): f \in \mathcal F\}\bigr| =
2^d$. The proof of the next technical lemma is a quite straightforward
generalization of similar bounds for the i.i.d.\ case. The analysis is
based on localization techniques for empirical processes.  We refer, for
instance, to \citep*[Corollary 3.7]{bartlett2005local} and to \citep{bousquet2019fast} for some similar results in the context of {\sc vc} classes.

\begin{lemma}
\label{lem:vclemma}
Let $X_1, \ldots, X_n$ be independent but not necessary identically distributed random variables taking their values in $\mathcal X$. Assume that the class $\F$ of $\{0, 1\}$-valued functions has the {\sc vc} dimension $d$. Then there are numerical constants $\kappa_1, \kappa_2 > 0$ such that for any $\delta \in (0, 1)$, with probability at least $1 - \delta$, for all $f \in \mathcal{F}$,
\begin{equation}
\label{eq:firstuniformbound}
\left|\sum_{i = 1}^n(f(X_i) - \E f(X_i))\right| \le \kappa_1\left(\sqrt{\left(\sum_{i = 1}^n\E f(X_i)\right)\left(d\log \frac{n}{d} + \log \frac{1}{\delta}\right)}+ d\log \frac{n}{d} + \log \frac{1}{\delta}\right)
\end{equation}
and
\begin{equation}
\label{eq:secuniformbound}
\left|\sum_{i = 1}^n(f(X_i) - \E f(X_i))\right| \le \kappa_2\left(\sqrt{\left(\sum_{i = 1}^n f(X_i)\right)\left(d\log \frac{n}{d} + \log \frac{1}{\delta}\right)}+ d\log \frac{n}{d} + \log \frac{1}{\delta}\right)~.
\end{equation}
\end{lemma}

\begin{proof}
  Without loss of generality we may assume that $0 \in \mathcal{F}$
  since by adding $f\equiv 0$ to the class the
  {\sc vc} dimension increases by at most one which can be absorbed
  by choosing slightly larger values of $\kappa_1, \kappa_2 > 0$. Consider the
  star-shaped hull of $\mathcal{F}$ around zero, that is, the class
  $\mathcal{H}$ of $[0,1]$-valued functions defined as
\[
\mathcal{H} = \{\alpha f: f \in \mathcal{F}, \alpha \in [0, 1]\}~.
\]
For $h \in \mathcal{H}$, we denote $Ph^2 = \frac{1}{n}\sum_{i = 1}^n\E h(X_i)^2$. Fix any $\delta \in (0, 1)$ and consider the  fixed point
\[
\gamma(\lambda, \delta) = \inf\left\{s > 0: \PROB\left(\sup\limits_{h \in \mathcal{H}, Ph^2 \le s^2}\left|\sum_{i = 1}^n(h(X_i) - \E h(X_i))\right| \le \lambda ns^2\right) \ge 1 - \delta\right\}~,
\]
where $\lambda > 0$ is a numerical constant specified below. 
By the definition of $\gamma(\lambda, \delta)$ we have, with probability at least $1 - \delta$,
\begin{equation}
\label{eq:uniformboundfiixedpoint}
\sup\limits_{h \in \mathcal{H}, Ph^2 \le \gamma(\lambda, \delta)^2}\left|\sum_{i = 1}^n(h(X_i) - \E h(X_i))\right| \le \lambda n \gamma(\lambda, \delta)^2~.
\end{equation}
Fix any $h \in \mathcal{H}$ such that
$P h^2 \ge \gamma(\lambda, \delta)^2$. Since $\mathcal{H}$ is
star-shaped, we have that
$h^{\prime} = h\gamma(\lambda, \delta)/\sqrt{Ph^2} \in \mathcal{H}$
and $P(h^{\prime})^2 = \gamma(\lambda, \delta)^2$, which, applying
\eqref{eq:uniformboundfiixedpoint} for $h^{\prime}$, implies on the
same event (and the same holds simultaneously for any such $h$)
\[
\left|\sum_{i = 1}^n(h(X_i) - \E h(X_i))\right| \le \lambda n\gamma(\lambda, \delta)\sqrt{Ph^2}~.
\]
The last inequality, combined with \eqref{eq:uniformboundfiixedpoint}, implies
\begin{equation}
\label{eq:uniformboundmain}
\sup\limits_{h \in \mathcal{H}}\left|\sum_{i = 1}^n(h(X_i) - \E h(X_i))\right| \le  \lambda n \gamma(\lambda, \delta)\sqrt{Ph^2} + \lambda n \gamma(\lambda, \delta)^2~.
\end{equation}
Finally, we need to prove an upper bound for $\gamma(\lambda,
\delta)$. Denoting $\mathcal{H}^{\prime} =
\mathcal{H}\cup(-\mathcal{H})$, we have
\[
\sup\limits_{h \in \mathcal{H}, Ph^2 \le s^2}\left|\sum_{i = 1}^n(h(X_i) - \E h(X_i))\right| = \sup\limits_{h \in \mathcal{H}^{\prime}, Ph^2 \le s^2}\left(\sum_{i = 1}^n(h(X_i) - \E h(X_i))\right)~.
\]
By \citep[Theorem 3.3.16]{gine2016mathematical} (see inequality (3.128) there which is relaxed in what follows by using
$\sqrt{2(2\E Z + \mathcal V_n)x} \le \sqrt{2\mathcal V_n x} + x + \E Z$), 
since almost surely $|h(X_i) - \E h(X_i)| \le 1$ and by fixing $x = \log \frac{1}{\delta}$, we have, with probability at least $1 - \delta$,
\begin{align}
&\sup\limits_{h \in \mathcal{H}^{\prime}, Ph^2 \le s^2}\left(\sum_{i = 1}^n(h(X_i) - \E h(X_i))\right) \nonumber
\\
&\le 2\E\sup\limits_{h \in \mathcal{H}^{\prime}, Ph^2 \le s^2}\left(\sum_{i = 1}^n(h(X_i) - \E h(X_i))\right) + s\sqrt{n\log\frac{1}{\delta}}+(5/2)\log \frac{1}{\delta}~.
\label{eq:highprobbound}
\end{align}
Finally, using the symmetrization inequality \citep{ledoux2013probability} we have
\[
\E\sup\limits_{h \in \mathcal{H}^{\prime}, Ph^2 \le s^2}\left(\sum_{i = 1}^n(h(X_i) - \E h(X_i))\right) \le 2\E\sup\limits_{h \in \mathcal{H}^{\prime}, Ph^2 \le s^2}\left(\sum_{i = 1}^n\varepsilon_i h(X_i)\right)~,
\]
where $\varepsilon_1,\ldots,\varepsilon_n$ are i.i.d.\ Rademacher
random variables with $\PROB\{\varepsilon_i=1\} =
\PROB\{\varepsilon_i=- 1\} =1/2$.
Conditioning on $X_{1}, \ldots, X_n$, we may use Dudley's entropy
integral bound
 (see, for instance, \citep{boucheron2013concentration}). First, we estimate
 the covering numbers of the set $\mathcal H$ with respect to the
 (random) distance $\rho(f, g) = \sqrt{\sum_{i = 1}^n(f(X_i) -
   g(X_i))^2/n}$. 
Denote 
\[
\text{diam}(n, s) = \sup\limits_{f, h \in \mathcal{H}, Ph^2 \le s^2}\rho(f, h)~.
\]
By the bound of \citet{haussler1995sphere}, the covering number of $\mathcal F$ at scale $r$ is upper bounded by $e(d + 1)\left(\frac{2e}{r^2}\right)^d$ and by a standard argument we have that the covering number of $\mathcal H$ is upper bounded by $e(d + 1)\left(\frac{8e}{r^2}\right)^d\left(1 + \lceil\frac{2}{r}\rceil\right)$ (see \citep[Proof of Corollary 3.7]{bartlett2005local}). 
Therefore, by the Dudley's bound we have, for some constants $c_1, c_2 > 0$,
\begin{align*}
&\E\sup\limits_{h \in \mathcal{H}^{\prime}, Ph^2 \le s^2}\left(\sum_{i = 1}^n\varepsilon_i h(X_i)\right) = \E\sup\limits_{h \in \mathcal{H}, Ph^2 \le s^2}\left|\sum_{i = 1}^n\varepsilon_i h(X_i)\right| \le c_1\sqrt{n}\E\int\limits_{0}^{\text{diam}(n, s)}\sqrt{d\log \frac{e}{r}}dr 
\\
&\le c_2\sqrt{n}\E\text{diam}(n, s)\sqrt{d\log \frac{e}{\text{diam}(n, s)}}\left(\IND_{\text{diam}(n, s) \ge \sqrt{d/n}} + \IND_{\text{diam}(n, s) < \sqrt{d/n}}\right) 
\\
&\le c_2\left(\sqrt{n}\E\text{diam}(n, s)\sqrt{d\log \frac{n}{d}} + d\sqrt{\log \frac{n}{d}}\right)~.
\end{align*}
By Jensen's inequality combined with the standard symmetrization and contraction inequalities \citep{ledoux2013probability} we have, for some $c_3 > 0$,
\begin{align*}
\sqrt{n}\E\text{diam}(n, s) &\le \sqrt{2\E\sup\limits_{h \in \mathcal{H}^{\prime}, Ph^2 \le s^2}\sum_{i = 1}^{n} h^2(X_i)}
\\
&\le \sqrt{2\E\sup\limits_{h \in \mathcal{H}^{\prime}, Ph^2 \le s^2}\sum_{i = 1}^{n} (h^2(X_i) -\E h^2(X_i)) + 2ns^2} 
\\
&\le c_3\left(\sqrt{\E\sup\limits_{h \in \mathcal{H}^{\prime}, Ph^2 \le s^2}\left(\sum_{i = 1}^n\varepsilon_i h(X_i)\right)}+ \sqrt{n}s\right)~.
\end{align*}
Combining the last two arguments, we have, for some $c_4 > 0$,
\begin{equation}
\label{eq:expectbound}
\E\sup\limits_{h \in \mathcal{H}^{\prime}, Ph^2 \le s^2}\left(\sum_{i = 1}^n\varepsilon_i h(X_i)\right) \le c_4\left(s\sqrt{dn\log \frac{n}{d}} + d\sqrt{\log \frac{n}{d}}\right)~.
\end{equation}
Finally, combining \eqref{eq:highprobbound}, \eqref{eq:expectbound}
and adjusting the constant $\lambda$ we have, for some $c_5 > 0$, 
that 
\[
\gamma(\lambda, \delta) \le c_5\sqrt{\frac{d\log \frac{n}{d} + \log \frac{1}{\delta}}{n}},
\]
which implies our first bound \eqref{eq:firstuniformbound} by \eqref{eq:uniformboundmain}.

To prove \eqref{eq:secuniformbound} we use that for $a, b, x > 0$,  $\sqrt{ab} \le \frac{a}{2x} + \frac{bx}{2}$. This implies
\begin{equation}
\label{eq:usefuleq}
\kappa_1\sqrt{\left(\sum_{i = 1}^n\E f(X_i)\right)\left(d\log \frac{n}{d} + \log \frac{1}{\delta}\right)} \le \frac{1}{2}\sum_{i = 1}^n\E f(X_i) + \frac{\kappa_1^2}{2}\left(d\log \frac{n}{d} + \log \frac{1}{\delta}\right)~,
\end{equation}
which, by \eqref{eq:firstuniformbound}, implies that on the same event where \eqref{eq:firstuniformbound} holds,
\[
\frac{1}{2}\sum_{i = 1}^n\E f(X_i) \le \sum_{i = 1}^n f(X_i) + \left(\kappa_1^2/2 + \kappa_1\right)\left( d\log \frac{n}{d} + \log \frac{1}{\delta}\right)~.
\]
Plugging this into \eqref{eq:firstuniformbound} and adjusting the constant $\kappa_2$ proves \eqref{eq:secuniformbound}.
\end{proof}

\section{Proof of Proposition \ref{prop:silviosbound}}

To simply the presentation we assume that the values $\log n, n^{1/3}, n^{1/6}, \ldots$ corresponding to the indexes are always integers.
It follows from Theorem \ref{thm:adaptiveestimator} that there exists a
constant $C > 0$ (which only depends on $\beta$ and $\phi(0)$) such that, with the same probability of error,
the adaptive estimator has an error at most
\[
   C \min \left( \frac{\sqrt{n} \log^{3/2} n}{\sum_{i > C\sqrt{n\log n}} \frac{1}{\sigma_i}}, \ \sigma_m \right)~,
\]
where $m$ is any integer that satisfies
\begin{equation}
\label{eq:mcondmain}
m\ge C\max\left( \sqrt{\sigma_m \sum_{i\ge m} \frac{1}{\sigma_i}
        \log n},  \log n \right)~.
\end{equation}
Therefore, it is sufficient to prove that for all sequences $\sigma_i$, 
\[
   \min \left(\frac{\sqrt{n} \log^{3/2} n}{\sum_{i > C\sqrt{n\log n}} \frac{1}{\sigma_i}}, \ \sigma_m \right)  
     \lesssim \sqrt{n}(\log^{3/2} n) \sigma_{C \log n}~.
\]
If 
\[
   \frac{\sqrt{n} \log^{3/2} n}{\sum_{i > C\sqrt{n\log n}}
     \frac{1}{\sigma_i}} \le \sqrt{n} (\log^{3/2} n) \sigma_{C \log n}~,
\]
then we are done, so we may assume
\[
   \frac{\sqrt{n} \log^{3/2} n}{\sum_{i > C\sqrt{n\log n}}
     \frac{1}{\sigma_i}}> \sqrt{n} (\log^{3/2} n) \sigma_{C \log n}~,
\]
or, equivalently,
\begin{equation}
\label{median}
   \sum_{i > C\sqrt{n\log n}} \frac{1}{\sigma_i} < \frac{1}{\sigma_{C\log n}}~. 
\end{equation}
It suffices to show that, when (\ref{median}) holds, then there exists
a value of $m$ satisfying (\ref{eq:mcondmain}) for which $\sigma_m \le \sqrt{n}(\log^{3/2} n) \sigma_{C\log n}$.

For any $m\le C\sqrt{n\log n}$, we may write
\[
    \sum_{i>m} \frac{1}{\sigma_i} = \sum_{i>C\sqrt{n\log n}} \frac{1}{\sigma_i} + \sum_{i\in [m, C\sqrt{n\log n}]} \frac{1}{\sigma_i}~.
\]
Using (\ref{median}), we see that $m$ satisfies (\ref{eq:mcondmain}) whenever
\begin{equation}
\label{eq:mcondsecond}
     \frac{m^2}{C^2\sigma_m} \ge \max\left( \frac{\log^2n}{\sigma_m}, \
     \frac{\log n}{\sigma_{C \log n}} + \log n \sum_{i\in [m,
       C\sqrt{n\log n}]} \frac{1}{\sigma_i} \right)~.
\end{equation}
First, note that if the first term dominates on the right-hand side of the above inequality, then $m = C\log n$ satisfies the inequality
above, and therefore the new bound is at most $C\sigma_{C\log n}$ and
our claim follows.

Hence, we may assume that the second term dominates and therefore we look for the values of $m$
such that
\begin{equation}
\label{eq:mcond1}
     \frac{m^2}{C^2\sigma_m} \ge 
     \frac{\log n}{\sigma_{C\log n}} + \log n \sum_{i\in [m,
       C\sqrt{n\log n})} \frac{1}{\sigma_i}~.
\end{equation}
We distinguish two cases depending on which term dominates on the
right-hand side: in case (i), 
\[
     \frac{1}{\sigma_{C\log n}} >  \sum_{i\in [m,
       C\sqrt{n\log n})} \frac{1}{\sigma_i}~,
\]
while in case (ii) the opposite holds.
In case (i), the right-hand side of \eqref{eq:mcond1} is at most $2\log n/\sigma_{C\log n}$.
Hence, we may take $m=C\log n$ to satisfy the inequality \eqref{eq:mcondmain} for $n$ large enough, leading to the bound
$C\sigma_{C\log n}$ which proves our claim.

In case (ii), the right-hand side of (\ref{eq:mcond1}) is bounded by 
\begin{equation}
\label{eq:mcond2}
  2 \log n \sum_{i\in [m,
       C\sqrt{n\log n})} \frac{1}{\sigma_i}
   \le 2 C \sqrt{n}(\log^{3/2} n) \frac{1}{\sigma_m}~.
\end{equation}
This implies by \eqref{eq:mcond1} that the inequality \eqref{eq:mcondmain} is satisfied when
\[
   m  \ge   \sqrt{2} C^{3/2} n^{1/4} (\log^{3/4} n)~.
\]
Since for $n$ large enough
\[
  n^{1/3} \ge  \sqrt{2} C^{3/2} n^{1/4} (\log^{3/4} n)~,
\]
this yields the upper bound $\sigma_{m_1}$ with $m_1= n^{1/3}$.

If $\sigma_{m_1} \le \sqrt{n} (\log^{3/2} n) \sigma_{C\log n}$, then the proof is finished. 
Otherwise, 
\begin{equation}
\label{eq:mcond3}
\sum_{i\in [m,
       C\sqrt{n\log n})} \frac{1}{\sigma_i} 
   \le \sum_{i\in [m,m_1)} \frac{1}{\sigma_i} + C\sqrt{n\log n} \frac{1}{\sigma_{m_1}} 
   \le \sum_{i\in [m,m_1)} \frac{1}{\sigma_i} +  \frac{C}{(\log n) \sigma_{C\log n}}~.
\end{equation}
Plugging this back to (\ref{eq:mcond1}), we see that in case (ii), the
upper bound becomes $\sigma_m$ for any $m$ that satisfies 
\begin{equation}
\label{eq:mcond4}
     \frac{m^2}{C^2\sigma_m} \ge 
     \frac{C+\log n}{\sigma_{C\log n}} + \log n \sum_{i\in [m, m_1)} \frac{1}{\sigma_i}~.
\end{equation}
This has the same form as (\ref{eq:mcond1}) but with a reduced range
in the summation on the right-hand side.

We proceed the same way as above. Once again, we consider two cases. In case (iii), 
\[
     \frac{C+\log n}{\sigma_{C\log n}} > \log n \sum_{i\in [m,
       m_1)} \frac{1}{\sigma_i}~,
\]
while in case (iv),
\[
     \frac{C+\log n}{\sigma_{C\log n}} \le \log n \sum_{i\in [m,
       m_1)} \frac{1}{\sigma_i}~,
\]
In case (iii), the right-hand side of (\ref{eq:mcond4})  is at most
$2(C + \log n)/\sigma_{C\log n}$, so, just like before,
we may take $m=C\log n$ to satisfy the inequality (\ref{eq:mcond1}), leading to the bound
$C\sigma_{C \log n}$ whenever $\log n \gtrsim C$.

In case (iv), the right-hand side of (\ref{eq:mcond4}) is bounded by 
\begin{equation}
\label{eq:mcond5}
   2 \log n\sum_{i\in [m,
       m_1)} \frac{1}{\sigma_i}
   \le 2 \log n \frac{m_1}{\sigma_m}= \frac{2n^{1/3} \log n}{\sigma_m}~.
\end{equation}
Thus, in this case \eqref{eq:mcondsecond} is satisfied for any $m\ge 2C
n^{1/6} \log^{3/2} n$, and in particular, for $m_2=n^{2/9}$. 
If $\sigma_{m_2} \le \sqrt{n} (\log^{3/2} n) \sigma_{C\log n}$, then the proof is finished. 
Otherwise, 
\[
\sum_{i\in [m,m_1)} \frac{1}{\sigma_i} 
   \le \sum_{i\in [m,m_2)} \frac{1}{\sigma_i} + \frac{m_1}{\sigma_{m_2}} 
   \le \sum_{i\in [m,m_2)} \frac{1}{\sigma_i} +  \frac{1}{n^{1/6} (\log^{3/2} n) \sigma_{C\log n}}~.
\]
Resubstituting into (\ref{eq:mcond1}), we see that in case (iv), the
upper bound becomes $\sigma_m$ for any $m$ that satisfies 
\[
     \frac{m^2}{C^2\sigma_m} \ge 
     \frac{C + \log n}{\sigma_{C\log n}} + \frac{1}{n^{1/6} (\log^{1/2} n) \sigma_{C\log n}} + \log n \sum_{i\in [m, m_2)} \frac{1}{\sigma_i}~.
\]
We may now continue the same fashion, at each step reducing the range
of the sum on the right-hand side unless at the $j$-th iteration $\sigma_{m_j} \le \sqrt{n} (\log^{3/2} n) \sigma_{C\log n}$ and we are done. In general, at the $j$-th
iteration, the summation is between $m$ and $m_j=n^{(2/3)^j/2}$. If we reach the
$j$-th iteration such that $m_j=C\log^{1/2} n$, we have
\[
\log n \sum_{i\in [m, m_j)} \frac{1}{\sigma_i} \le \frac{C\log^{3/2} n}{\sigma_m}~,
\]
so that one may choose $m = C\log n$ for large enough $n$. The claim follows.
\bibliography{mybib}

\end{document}